\documentclass[12pt]{amsart}
\usepackage{amsmath}
\usepackage{amssymb}
\usepackage{amsthm}
\usepackage{epsfig,graphicx}
\usepackage{mathrsfs}

\newcommand{\ignore}[1]{\relax}

\newcommand{\im}{\operatorname{Im}}

\newcommand{\C}{\mathbb C}

\newcommand{\R}{\mathbb R}
\newcommand{\Z}{\mathbb Z}

\newcommand{\SSS}{\mathcal S}

\newcommand{\re}{\operatorname{Re}}

\newtheorem{lem}{Lemma}[section]
\newtheorem{claim}{Claim}[section]

\newtheorem{theorem}{Theorem}
\newtheorem{corollary}[lem]{Corollary}
\newtheorem{thm}[lem]{Theorem}

\newtheorem{coro}[lem]{Corollary}
\newtheorem{prop}[lem]{Proposition}

\theoremstyle{definition}
\newtheorem{condition}[claim]{Condition}
\newtheorem{defn}[lem]{Definition}

\newtheorem{problem}[claim]{Problem}
\newtheorem{exa}[lem]{Example}
\newtheorem{que}[lem]{Question}

\theoremstyle{remark}
\newtheorem{rmk}[lem]{Remark}
\newtheorem{rem}[lem]{Remark}

\renewcommand{\setminus}{\smallsetminus}

\newcommand{\tordva}{(\C^\times)^{2}}

\newcommand{\rtordva}{(\R^\times)^{2}}

\newcommand{\conj}{\operatorname{conj}}

\newcommand{\dd}{\partial}

\newcommand{\rp}{{\mathbb R}{\mathbb P}}

\newcommand{\pp}{{\mathbb P}}

\renewcommand{\setminus}{\smallsetminus}

\newcommand{\MM}{{\mathcal M}}

\newcommand{\ar}{\operatorname{Area}}

\ignore{
\newtheorem{theorem}{Theorem}

\newtheorem{condition}[theorem]{Condition}

\newtheorem{corollary}[theorem]{Corollary}

\newenvironment{proof}[1][Proof]{\noindent\textbf{#1.} }{\ \rule{0.5em}{0.5em}}
}

\renewcommand{\setminus}{\smallsetminus}

\begin{document}
%\title{Half-dimensional amoebas.}
\title
%[Quantum indices of real plane curves]
%{Canonical embedding of a separating $(M-2)$-curve is hyperbolic}
{Area in real K3-surfaces}
\author{Ilia Itenberg and Grigory Mikhalkin}
\address{
Institut de Math\'{e}matiques de Jussieu--Paris Rive Gauche\\
Sorbonne Universit\'{e}\\ 
4 place Jussieu, 75252 Paris Cedex 5, France} 
\email{ilia.itenberg@imj-prg.fr}
\address{Universit\'e de Gen\`eve,  Math\'ematiques, Battelle Villa, 1227 Carouge, Suisse}
\email{grigory.mikhalkin@unige.ch} 
\thanks{Research is supported in part by 
the SNSF-grants 
178828, 182111  and NCCR SwissMAP (G.M.),
and by the ANR grant ANR-18-CE40-0009 ENUMGEOM (I.I.)}  

\begin{abstract}
For a real K3-surface $X$, one can introduce areas of connected components
of the real point set $\R X$ of $X$ using a holomorphic symplectic form of $X$.
These areas are defined up to simultaneous multiplication by a
positive real number, 
so the areas of different components can be compared.
In particular, it turns out that the area of a non-spherical component of  $\R X$
is always greater than the area of any spherical component.

In this paper we explore further comparative restrictions on the area for real K3-surfaces
admitting a suitable polarization of degree $2g - 2$
(where $g$ is a positive integer) and such that $\R X$ has one non-spherical component
and at least $g$ spherical components.
For this purpose we introduce and study the notion of simple Harnack curves
in real K3-surfaces, generalizing planar simple Harnack curves from \cite{Mi00}. 
%we prove that 
%the area of the non-spherical component is greater than the area of every spherical component. 
%We also prove that, for any real K3-surface with a real component of genus at least two, 
%the area of this component is greater than the area of every spherical real component. 
\end{abstract} 

\maketitle

\section{Introduction}
A K3-surface $X$ is a smooth simply connected complex surface admitting
a holomorphic symplectic form, that is, a holomorphic 2-form $\Omega$ such that
$\Omega\wedge\bar\Omega$ is a volume form.
A K3-surface $X$ is called {\em real} if it comes with an anti-holomorphic 
involution $\sigma:X\to X$. The fixed point set of $\sigma$ is denoted with $\R X$
and called the {\em real locus} of $X$. If non-empty, $\R X$ is an orientable surface. 
All K3-surfaces are diffeomorphic, but their real loci may have
different topological types,
see \cite{Nik, K3book} for details. 

There are finitely many possibilities for the topological type of $\R X$.
For example, the surface $\R X$ may be diffeomorphic to the disjoint
union of two tori. In this case, we call $X$ a {\em hyperbolic} real K3-surface. 
The two components of the real locus of a hyperbolic K3-surface are homologous.
If $X$ is not hyperbolic, then $\R X$ has at most one non-spherical component.
Denote by $a$ the number of connected components of $\R X$,
and denote by $b$ the half of the first Betti number of $\R X$.
%consists of $a$ spherical components and 
%eventually a single 
%at most one non-spherical component.
%$a+b\le11$, the latter inequality being 
As a corollary of the Smith-Thom inequality \cite{Smith, Thom}, one obtains
$a+b\le12$. 
There are further restrictions on $a$ and $b$, namely we have
%$b-a\equiv1\pmod{8}$ if $a+b=11$, 
$a-b\equiv 0 \pmod{8}$ if $a+b=12$,
and
%$b-a\equiv 0,2\pmod{8}$ if $a+b=10$, 
$a-b\equiv \pm 1 \pmod{8}$ if $a+b=11$,  
according to 
the congruences on the Euler characteristic of the real locus of maximal (in the sense of the Smith-Thom inequality)
and submaximal real algebraic surfaces (see \cite{Rokhlin, Kha, GugKra}). 
%The topological classification of real K3-surfaces, as well as their 
The deformation classification of real K3-surfaces is, essentially, due 
to Nikulin \cite{Nik}. %and Kharlamov \cite{Kha}.  

Multiplying the holomorphic 2-form $\Omega$ by a non-zero complex number
we may assume that the closed real 2-form $\alpha=\re \Omega$ is invariant with respect to
the involution $\sigma$ while $\beta=\im \Omega$ is anti-invariant.  
The form $\alpha$ is non-vanishing, and thus defines an orientation on $\R X$.
Hence, $\alpha$ may be viewed as an area form on $\R X$, well-defined up to a real multiple. 
%\mnote{to say that $d\alpha = 0$, to mention the orientation issue,
%and not to use integral notation} 
Thus, we may compare total areas of different components of $\R X$. 
If $K\subset\R X$ is a component, then we denote with $\ar(K)>0$
the absolute value of $\int\limits_K\Omega=\int\limits_K\alpha$.
For instance, if $\R X$ is hyperbolic, {\em i.e.}, consists of two components $T_1$ and $T_2$,
each diffeomorphic to the 2-torus, then
\[
\ar(T_1)=\ar(T_2)
\]
since in this case the components $T_1$ and $T_2$ are homologous.

Suppose that there exists a smooth real curve $C\subset X$, that is, a smooth curve 
%of genus $g$
invariant with respect to the involution $\sigma$.
All (not necessarily smooth) real curves in $X$ linearly equivalent to $C$ 
form a linear system. By the adjunction formula,
%the projective rank of this system is $g-1$, i.e.
all such curves 
%form 
constitute the real projective space $\rp^{g}$, 
where $g$ is the genus of $C$. %(assuming that $g>0$).
%\mnote{It seems that `(assuming that $g>0$)'
%can be removed; Slava suggests not to speak about polarization if $g = 1$} 
Such linear system is called {\em polarization} if $g>0$ 
(we extend the standard terminology to the case $g=1$).
Accordingly, we say that the real K3-surface $X$ is genus $g$ {\em polarized}
%if it admits an embedding of a real smooth curve $C\subset X$ of genus $g$
%invariant with respect to the involution $\sigma$
if such a linear system (also called {\em polarization}) is fixed.
The square of the homology class $[C]\in H_2(X)$ is equal to $2g-2$ by
the adjunction formula, so we also say that such a polarization is {\em of degree $2g-2$}. 

The real locus $\R C=C\cap\R X$ is a smooth 1-dimensional manifold and therefore
diffeomorphic to the disjoint union of $l$ circles.
By the Harnack inequality \cite{Har}, we have $l\le g+1$.
Clearly, the homology class $[\R C]\in H_1(\R X;\Z_2)$
does not depend on the choice of the curve $C$ in the polarization.
We say that the polarization is {\em non-contractible} if $[\R C]\neq 0$.
In such case, $\R X$ must contain a non-spherical component, 
which we denote with $N\subset\R X$. %Let $b$ be the genus of $N$.
Unless $X$ is hyperbolic, 
all other components of $\R X$ are spheres. We denote them
with $S_j$, $j=1, \dots, a - 1$. 

\ignore{
A generic projective K3-surface $X$ has the Picard number equal to one.
In this case, the homology class of any curve in $X$ is proportional to $[C]\in H_2(X)$. 
In particular, 
%in this case the 
$X$ does not contain any embedded rational curves. 
Indeed, by the adjunction formula, the square of 
such a curve is $-2$
while $[C].[C]=2g-2\ge 0$, which contradicts to the proportionality.
We refer to the embedded rational curves as {\em $(-2)$-curves}.

Some special genus $g$ polarized K3-surfaces may admit $(-2)$-curves while
some of those $(-2)$-curves $E$ might be real. 
%i.e. invariant with respect to $\sigma$.
%As usual, denote $\R E=E\cap\R X$. Topologically, $\R E$ is a circle unless $\R E=\emptyset$. 
Denote with $R \subset H_1(\R X;\Z_2)$ the subgroup generated by real loci of all real
$(-2)$-curves.
We say that the polarization is {\em strongly non-contractible} if $[\R C] \notin R$, 
{\em i.e.}, $[\R C]$ cannot be presented as a sum of real loci of real $(-2)$-curves in $X$. 
%Clearly, any noncontractible polarization is strongly noncontractible if $X$ does not
%contain $(-2)$-curves.
%The polarization is {\em real} if we may choose $C\subset X$ to be invariant
%with respect to  
}

%The curve $C$ may be treated as a divisor in $X$. All curves  
The principal result of this paper is the following theorem.
\begin{theorem}\label{main}
Suppose that $X$ is a real K3-surface admitting a non-contractible 
genus $g>0$ polarization %with a real component $N$ of genus $b$,
and such that $\R X$ has $a - 1\ge g$ spherical components.
Then, we have 
\[
\ar(N)>\sum\limits_{j=1}^{a-g}\ar(S_j),
\]
where $N$ is the non-spherical component of $\R X$ and
$S_j$, $j=1,\dots a - 1$, are its spherical components.
\end{theorem} 

\ignore{
The standard description and facts concerning period spaces of polarized $K3$-surfaces (see, for example, \cite{Beauville})
imply that for any real K3-surface $X$ with a non-contractible 
genus $g>0$ polarization $|C|$, the pair $(X, |C|)$ can be approximated by a continuous family 
($X_s, |C_s|$), $s \in [0, 1]$, where, for every $s$, the pair $(X_s, |C_s|)$ 
is formed by a real K3-surface $X_s$ and a non-contractible 
genus $g>0$ polarization $|C_s|$ of $X_s$, the pair $(X_0, |C_0|)$ coincides with $(X, |C|)$, 
and the surface $X_s$ does not contain $(-2)$-curves for every $s \in (0, 1]$. 
The following corollary is an immediate consequence of Theorem \ref{main}
and the above mentioned fact.  

\begin{corollary}\label{main-corollary}
Suppose that $X$ is a real K3-surface admitting a non-contractible 
genus $g>0$ polarization and such that $\R X$ has at least $g$ spherical components.
Then, we have 
\[
\ar(N) \geq \ar(S),
\]
where $N$ is the non-spherical component of $\R X$ and
$S$ is any of its spherical components.
\end{corollary}
}

\section{Area inequalities from 
%arithmetics}
linear algebra} 
%In the case if $N$ is not a torus then the following strengthened version
%of Theorem 1 follow from simple arithmetical considerations in $H_2(X)$. 
%The holomorphic 2-form $\Omega=\alpha+i\beta$ can be used to compute
%the volume of $X$, we set
%\[
%\vol(X)=\frac14\int\limits_X\Omega\wedge\bar\Omega>0.
%\]
Consider a real K3-surface $X$ (in this section, we do not assume that $X$ is algebraic). 
Denote with $[A]\in H_2(X;\R)$ the homology class dual 
to the real 2-form $\alpha$.
By the Hodge-Riemann relations we have
\[ [A].[A]>0.\]

%\begin{prop}\label{positive} 
%If $X$ is a real K3-surface, $N\subset\R X$ is a component of genus $b>1$,
%and $S\subset\R X$ is a spherical component, 
%then 
%\[
%\ar^2(N)\ge (b-1)(\ar^2(S)+2[A].[A]).
%\]
%\end{prop}

\begin{prop}\label{positive} 
If 
%$X$ is a real K3-surface (not necessarily algebraic), 
$N\subset\R X$ is a component of genus $b>1$,
and $\Sigma_1, \ldots, \Sigma_k \subset\R X$ are the spherical components,
then 
\[
\ar^2(N) \ge (b-1)\left(\sum_{i = 1}^k \ar^2(\Sigma_i)+2[A].[A]\right).  
\]
\end{prop}

\begin{proof}
%[Proof of Proposition \ref{positive}]
We have the decomposition 
\[
H_2(X;\R)=H_2^+(X;\R)\oplus H_2^-(X;\R),
\]
where $H^+_2(X;\R)$ stands for the $\sigma$-invariant part of the vector space $H_2(X;\R)$
and $H^-_2(X;\R)$ for its anti-invariant part. 
The decomposition is orthogonal with respect to the intersection form on $H_2(X; \R)$.  
%The hyperplane section class, as well as the 
We have  $[A]\in H^+_2(X;\R)$, while 
the class dual to $\beta$ belongs to $H_2^-(X;\R)$.
%while $[A]\in H^+_2(X;\R)$. 
In addition, a K\"ahler form on $X$ can be chosen in such a way that the class
of this form belongs to $H_2^-(X;\R)$. 
Thus, the intersection form restricted to $H^+_2(X;\R)$ 
has one positive square. 

Consider the subspace $V\subset H^+_2(X;\R)$ generated by $[A]$, $[N]$, $[\Sigma_1]$, $\ldots$, $[\Sigma_k]$, 
where $N$, $\Sigma_1$, $\ldots$, $\Sigma_k$ are oriented by $\alpha$. 
The determinant of the intersection matrix 
%of $V$ in this basis is
of these vectors is 
\[
{\mathcal D} = 
\begin{vmatrix}
[A].[A] & \ar(N) & \ar(\Sigma_1) & \ar(\Sigma_2) & \ldots & \ar(\Sigma_k) \\
\ar(N) & 2(b-1) & 0 & 0 &   \ldots & 0 \\
\ar(\Sigma_1) & 0 & -2 & 0 &  \ldots & 0 \\
\ldots & \ldots & \ldots & \ldots & \ldots & \ldots \\ 
\ar(\Sigma_k) & 0 &0 &  \ldots & 0 & -2 
\end{vmatrix}, 
%=2\ar^2(N)-2(g-1)\ar^2(S)-4(g-1)[A].[A]>0.
\]
since the self-intersection of a component of $\R X$ in $X$ is minus
the Euler characteristic of the component. 
We get
\[
\displaylines{
(-1)^{k + 1}{\mathcal D} \cr
= 2^{k}\left(\ar^2(N)-(b-1)\sum_{i = 1}^k \ar^2(\Sigma_i) - 2(b-1)[A].[A]\right) \geq 0, 
} 
\]
since a diagonalization of the intersection form on $V\subset H^+_2(X;\R)$ 
%may contain no more than 
contains exactly one positive square. 
\end{proof} 

\begin{coro}\label{coro-arithm}
If 
%$X$ is a real K3-surface, 
$N\subset\R X$ is a component of genus $b>1$,
%and $S\subset\R X$ is a spherical component, 
and $\Sigma_1, \ldots, \Sigma_k \subset\R X$ are spherical components, 
where $k \leq b - 1$, 
then 
\[
\ar(N) > \sum_{i = 1}^k \ar(\Sigma_i).
\]
\end{coro}

\begin{proof}
The statement is an immediate corollary of Proposition \ref{positive} and the inequality
$$
n(x_1^2 + \ldots + x_n^2) \geq (x_1 + \ldots + x_n)^2
$$ 
valid for any integer $n \geq 1$ and any real numbers $x_1$, $\ldots$, $x_n$. 
\end{proof} 

\begin{rmk}
In the first version of our paper, the main
result was
%application of the theory of simple Harnack
%curves in K3-surfaces was
a weaker version of Theorem \ref{main},
%namely
namely, only the inequality $\ar(N)>\ar(S_1)$ under similar hypotheses. %similar to those of T 
In view of Corollary \ref{coro-arithm} this inequality is only non-trivial if the genus of $N$ is 1.
However, the referee of our paper suggested an elegant simple proof of this inequality 
based on so-called {\em Donaldson's trick} (which consists in changing the complex structure on a K3-surface so that an anti-holomorphic involution becomes holomorphic; see \cite{Don} or \cite{K3book}) and 
applicable even for non-projective real K3-surfaces with a torus component.
We are strongly indebted to the referee for this remark which has pushed us to find
a stronger version of Theorem \ref{main} considered in this paper.
This stronger version still comes as an application of simple Harnack curves in 
a K3-surface that are studied in this paper. 
\end{rmk} 

%\begin{rmk}
Referee's suggestion mentioned in the previous remark can be generalized in the following way
strengthening Corollary \ref{coro-arithm}.
%\end{rmk}

\begin{prop}
If 
%$X$ is a real K3-surface, 
$N\subset\R X$ is a component of genus $b\ge 1$,
%and $S\subset\R X$ is a spherical component, 
and $\Sigma_1, \ldots, \Sigma_k \subset\R X$ are spherical components, 
where $k \leq b $, 
then 
\[
\ar(N) > \sum_{i = 1}^k \ar(\Sigma_i).
\]
\end{prop}
\begin{proof}
By the Calabi-Yau theorem, for any K\"ahler form on $X$, there exists a unique
K\"ahler-Einstein metric whose fundamental class coincides with the chosen form. 
Thus, $X$ admits a $\sigma$-anti-invariant K\"ahler form. 
Let $\omega$ be a $\sigma$-anti-invariant K\"ahler form on $X$ such that $[\omega]^2=[\alpha]^2 = [\beta^2] > 0$. 
Consider the hyperk\"ahler rotation 
%of $X$ 
that cyclically exchange the triple $\alpha,\beta,\omega$ 
(in such a way that $\alpha$ becomes a new K\"ahler form),
and denote the resulting K3-surface with $X'$. 
Since 
%the forms $\beta$ and $\omega$ vanished on $\R X$ before the rotation, 
the new holomorphic $2$-form on $X'$ 
%vanishes on $\R X$, and thus 
is $\sigma$-anti-invariant, the involution $\sigma$ is holomorphic on $X'$
and $\R X$ 
%is 
becomes a holomorphic curve in $X'$ (for details, see \cite{Don} or \cite{K3book}). 

The component $N$ is a holomorphic curve in $X'$ 
%defining 
and defines a polarisation of genus $b$ of $X'$. 
Choose  a point $p_i \in \Sigma_i$ for every $i=1,\dots,k$. 
Since $k\le b$, there exists a holomorphic curve $N'$
from the polarisation such that $N'$ passes through $\{p_i\}_{i=1}^k$. 
For any $i = 1$, $\ldots$, $k$, 
the class $[N']=[N]\in H_2(X)$ is orthogonal
to $[\Sigma_i]$. Thus, $\Sigma_i$ is an irreducible component of $N'$.
It remains to notice that $\int\limits_{C'}\alpha > 0$ for any holomorphic curve $C' \subset X'$. 
\end{proof} 

%The rest of the paper is devoted to the proof of Theorem 1 in the case
%when $N$ is a torus. 
%To achieve this, we take a look at real curves in $X$. 

\section{Curves in $X$ and their deformations}
We take a closer look at the polarization ({\em i.e.}, the linear system $|C|\approx\pp^g$) defined 
by a smooth real curve $C \subset X$. Denote by $\MM\subset |C|$ 
the space of all
smooth curves linearly equivalent to $C$.
It is well-known that $\MM$ is a smooth manifold of dimension $g$.
The tangent space $T_C\MM$ 
%is given by 
consists of holomorphic normal vector fields to $C$ in $X$. 
The non-degenerate holomorphic 2-form $\Omega$ provides an identification between
$T_C\MM$ and the space 
%$\Omega(C)$ 
of holomorphic 1-forms on $C$ 
(through plugging into $\Omega$ the normal vector field corresponding to a vector
from $T_C\MM$).

Let $\widetilde\MM\to\MM$ be the universal covering consisting
of pairs $\tilde C' = (C', \gamma)$, 
where $C'\in\MM$ and $\gamma$ is a homotopy class of a path
connecting $C$ and $C'$ in $\MM$.
For $Z\in H_1(C)$ we
define the map $I_Z:\widetilde\MM\to\C$ by
\begin{equation}\label{mapIz}
I_Z(\tilde C')=\int\limits_{Z_\gamma}\Omega.
\end{equation}
Here $Z_\gamma$ is the surface spanned by a cycle from $Z$ under the monodromy from $\gamma$.
Since the closed 2-form $\Omega$ vanishes on any holomorphic curve (including $C$ and $C'$),
the value $I_Z(\tilde C')\in\C$ is well-defined.
%\mnote{Slava asks for additional argument} 

Let $a_1,\dots,a_g\in H_1(C)$ be a 
%basis 
maximal collection
%\mnote{`maximal collection'}  
of $a$-cycles, 
{\em i.e.}, a collection of linearly independent primitive elements 
with trivial pairwise intersection numbers.
\begin{lem}\label{intI}
The map  
\begin{equation}\label{mapI}
I=(I_{a_1},\dots,I_{a_g}):\widetilde\MM\to\C^g.
\end{equation}
is a local diffeomorphism.
\end{lem}
\begin{proof}
Since $I$ is a map between manifolds of the same dimension,
it suffices to prove that its differential is injective.
The kernel of $dI$ at $\tilde C'\in\widetilde\MM$
consists of holomorphic forms on $C'\in\MM$ with zero periods
along $a_1,\dots,a_g$. By the Riemann theorem, any such holomorphic 
form on $C'$ must vanish. 
\end{proof} 

\begin{rmk}
Lemma \ref{intI} is a shadow of the so-called Beauville-Mukai
integrable system (see \cite{Beauville}) on the universal Jacobian over $\widetilde\MM$. 
The maps $I_Z$ are the integrals of this system.
%\mnote{reference?} 
\end{rmk} 

The system $(a_1,\dots,a_g)$ of $a$-cycles can be represented with 
a system $a$ of $g$ pairwise disjoint simple 
%closed curves 
loops on $C$. Their complement in $C$
is a sphere with $2g$ holes.
Let $\widetilde\MM^a$ be the space
consisting of pairs $(C', [\gamma])$, such that 
$C'\in |C|$, is (at worst) a nodal curve,
and
$[\gamma]$ is a homotopy class of a path $\gamma:[0,1]\to |C|$ such that $\gamma(0)=C$, 
$\gamma(1)=C'$,
%$\gamma^{-1}(\MM)
%=
%\supset [0,1)$
%\mnote{The definition of $\MM^a$ is changed
%\endgraf It: with this definition, there is no inclusion $\widetilde\MM^a\supset\widetilde\MM$} 
and for all $t\in [0,1]$ the curve $\gamma(t)$ is at worst a nodal curve
whose vanishing cycles %of the pair $(C', \gamma)$
under the monodromy
are represented by simple 
%closed curves 
loops on $C$ 
disjoint with the family $a$.
Note that the forgetting map
\[
\widetilde\MM^a\to |C|,\ (C', [\gamma])\mapsto C'
\]
is a local diffeomorphism.
The definition of the map \eqref{mapI} 
%uniquely 
naturally extends to the map 
\begin{equation}\label{mapIa}
I^a:\widetilde\MM^a\to\C^g.
\end{equation}
Lemma \ref{intI} extends to the following proposition.
\begin{prop}\label{intIa}
The map %\eqref{mapIa} 
$I^a$ is a local diffeomorphism.
\end{prop}
\begin{proof}
For a nodal curve $C'\in |C|$ the holomorphic 2-form gives an identification
between $T_{C'}|C|$ and the space of meromorphic forms on the normalization
of $C'$ with at worst simple poles over the nodes such that the residues
at the two preimages of the same node are opposite.
The space of such forms is $g$-dimensional. 
The kernel of $dI^a$ over $C'$ consists of the forms with zero periods over $a_j$, $j=1,\dots,g$,
and thus trivial.
\end{proof}

%More generally, let $\widetilde\MM^a_s\supset\widetilde\MM^a\supset\widetilde\MM$ be the space 
%consisting of pairs $(C', [\gamma])$, such that $C'\in |C|$ 
%and $[\gamma]$ is a homotopy class of a path $\gamma:[0,1]\to |C|$ verifying the following properties: 
%\begin{itemize}
%\item $\gamma(0)=C$ and $\gamma(1)=C'$, 
%\item for all but finitely many values of $t \in [0, 1]$ the curve $\gamma(t)$
%belongs to $\MM$, 
%\item the system $a$ can be extended to a continuous family $a_t$, $t \in [0, 1]$,
%such that, for any $t \in [0, 1]$, the system $a_t$ is formed by $g$ simple loops 
%on the curve $\gamma(t)$, and all loops in $a(t)$ are disjoint from the singular locus of $\gamma(t)$ 
%(and, in particular, disjoint from the multiple components in the case of non-reduced curve $\gamma(t)$).  
%\end{itemize}
%
%The map \eqref{mapI} uniquely extends to
%\begin{equation}\label{mapIas} 
%I^a_s:\widetilde\MM^a_s\to\C^g.
%\end{equation} 

\section{Simple Harnack curves and their degenerations}
Simple Harnack curves in toric surfaces were introduced and studied in \cite{Mi00}.
A toric surface $Y\supset\tordva$ may be considered as a {\em log K3-surface}, or a K3-surface 
relative to its toric divisor $D=Y\setminus\tordva$. %(the complement of the dense toric orbit of $Y$).
Indeed, $D$ is the pole divisor for the meromorphic extension
of the holomorphic form $dz_1\wedge dz_2$ on $\tordva$.
%In particular, Lemma \ref{intI} admits a straightforward modification for the space
%$\tilde\MM(\Delta)$.
In this section, we define and study counterparts of these curves in (closed) K3-surfaces. 

%Recall that an irreducible (over $\C$) real curve is called an {\em M-curve} if the number of 
%the real components of its normalization equals to one plus its genus
%(i.e. the normalization has the maximal number of real components allowed
%by the Harnack inequality). 
Recall that a smooth (and irreducible over $\C$) real curve $C$ is called an {\em M-curve} 
(or a {\em maximal curve}), 
if the number of 
its real components is equal to one plus its genus
({\it i.e.}, if it has the maximal number of real components allowed 
by the Harnack inequality). 
An M-curve $C$ is {\em dividing}, {\it i.e.}, $C\setminus\R C$ consists of two components interchanged 
by the involution of complex conjugation. 

An orientation on the real locus $\R C$ of a dividing curve $C$
is called the {\em complex orientation} if it comes as the boundary
orientation of a component of $C\setminus\R C$, see \cite{Ro}.
Clearly there are two complex orientations on $\R C$ and they are opposite.
%Below by the complex orientation we mean one of these two.

\begin{defn}\label{dH}
A smooth real M-curve $C\subset X$ 
%in a closed real surface $X$
is called {\em simple Harnack}
%if it is an M-curve, 
%and
%if for any two connected components $L,L'\subset\R C$
if for any component $K\subset\R X\setminus\R C$
and any two distinct components $L,L'\subset\R C$ adjacent to $K$
%is orientable and
%we have the following property.
%{\em Either the closure of $K$ coincides with the 
%contains more than one component of $\R C$ then $\Sigma\setminus \R 
a complex orientation of $L$ and $L'$
%induced as the boundary orientation
%of a half of $C\setminus\R C$
can be extended to an orientation of $K$.
%the boundary orientation
%of $\dd K\subset\R C$ induced by an orientation of $K$ agrees
%with the boundary orientation on $\R C$
%induced by an orientation of a component of $C\setminus\R C$.
%by the complex orientation of one of the halves of $C\setminus\R C$.
\end{defn}
Note that this definition allows for two types of components %$\Sigma\subset\R X$
$K\subset \R X\setminus\R C$.
Either we have $\dd K=\dd\bar K$ for the closure $\bar K\subset\R X$, 
%(i.e. no component of $\R C$ is included to $\dd K$ more than once)
or the closure $\bar K$ is the entire connected component $\Sigma\subset\R X$.
%to contain a single non-contractible 
%({\em i.e.}, realising a non-zero class in $H_1(\R X;\Z_2)$) 
%(in $H_1(S)$) 
%(in $\Sigma$) 
%component of a simple Harnack curve. 
%However, if  $\Sigma$ contains at least two components of $\R C$, then
%all these components must be contractible. 
In the first case, the complex orientations of the components of $\R C$ %of $\Sigma\cap\R C$
must alternate as in 
Figure \ref{orient}. 
\begin{figure}[h]
\includegraphics[height=45mm]{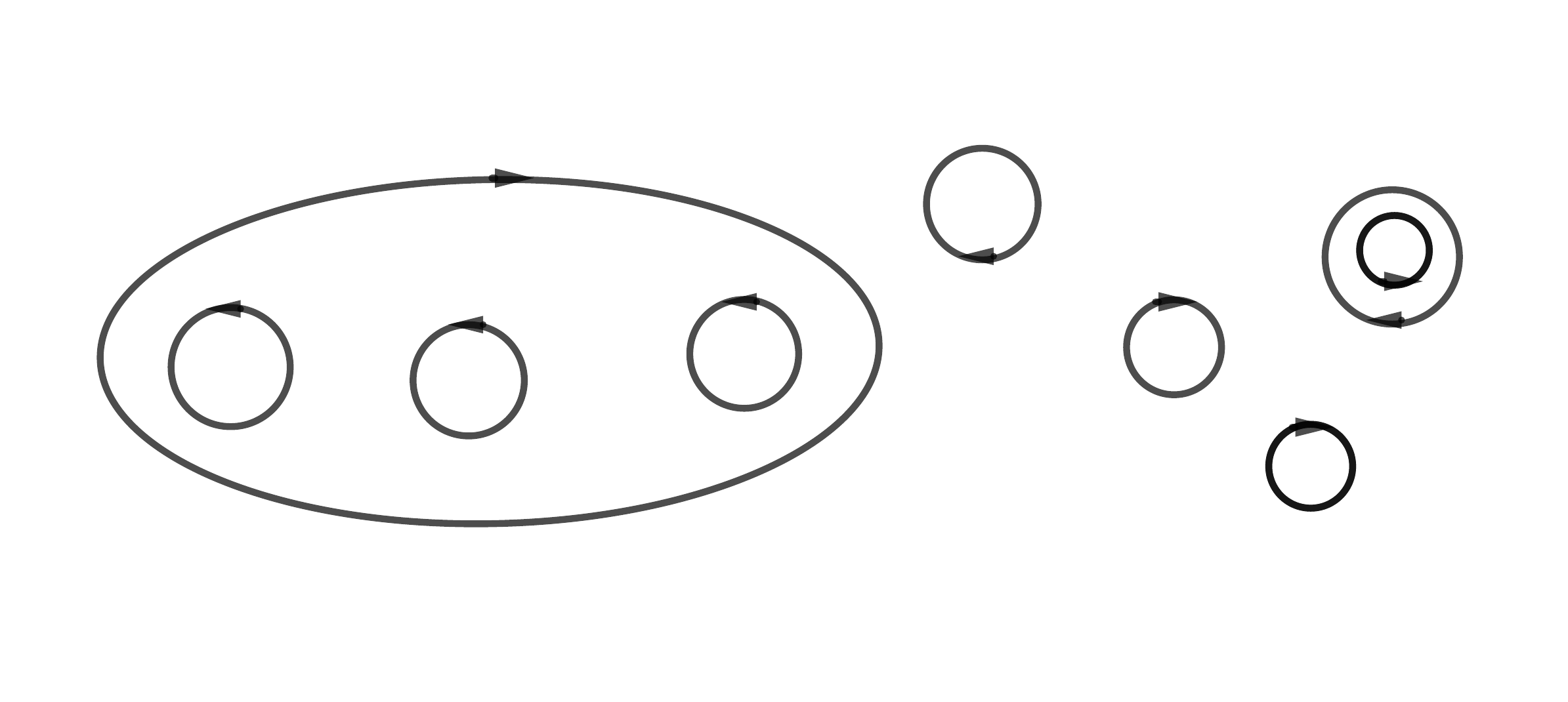}\vspace{-20pt}
\caption{Orientations imposed by Definition \ref{dH}.
\label{orient}}
\end{figure}
In the second case, the component $\Sigma\subset\R X$ contains a single component
$L\subset\R X$ and $[L]\neq 0\in H_1(\R X;\Z_2)$.
%We call the components corresponding to the second case {\em $m$-components}.
\begin{defn}
A component $L\subset\R C$ is called a {\em modifiable} component of $\R C$ or an
{\em $m$-component} if $[L]\neq 0\in H_1(\R X;\Z_2)$ and the component of $\R X$
containing $L$ does not contain any other component of $\R C$. 
%$\Sigma_j\cap\R X=L_j$
%for the component $\Sigma_j\subset\R X$ such that $\Sigma_j\cap L_j\neq\emptyset$.
%(In other words, $L_j$ is the only component of $\R C$ in $\Sigma_j$.)
\end{defn} 
Unless $X$ is a hyperbolic K3-surface, $\R X$ has not more than one non-spherical component
and thus $\R C$ may have not more than one $m$-component.

%For the purposes of this paper we need this definition for the case when $X$ is a K3-surfaces.
%Then orientability of components of $\R X\setminus\R C$ is automatic (as $\R X$ is orientable),
%but concordance of the boundary orientation is a non-trivial condition.
\begin{rem}
Simple Harnack curves in toric surfaces from \cite{Mi00} can be defined through a relative version
of Definition \ref{dH}. Namely, a real M-curve $C$ in a real toric surface $Y$ is simple Harnack
if $C\setminus \tordva=\R C\setminus \tordva$ ({\em i.e.}, all intersection points of $C$ and the toric divisor 
are real) and the orientation
of $\dd K\subset(\R C\cap\tordva)$ induced from $K\subset \rtordva\setminus\R C$
agrees with an orientation of a component of $C\setminus\R C$.
\end{rem}
\begin{rem}
It seems that Definition \ref{dH} also might be meaningful for the case when $X$ is a surface
%also 
different from a K3-surface. In a more general setting we add an assumption that 
each component of $\R X\setminus\R C$ is orientable (the condition that holds automatically
in the case of K3-surfaces thanks to the non-vanishing 2-form $\Omega$).  
\end{rem} 

We say that a curve $C_0 \subset X$ is a {\em degeneration} of simple Harnack curves
if there exists a continuous family $C_t\subset X$, $t\in [0,1]$, such that $C_t$ is a simple 
Harnack curve for every $t \in (0, 1])$. 
%and $C_0 = C$. 
Clearly, $C_0$ is a real curve which may develop some singularities. 
Also, the degeneration $C_0$ does not have to be irreducible, or
even reduced. It consists of several components while some of these components
may be taken with multiplicity greater than 1 (multiple components).
We refer to components of $C_0$ whose multiplicity equal to 1 as simple components of $C_0$. 
%Unless $C$ is non-reduced, it has isolated singularities. 

\begin{prop}\label{hK3}
Let $C_0 \subset X$ be a degeneration of simple Harnack curves.
Then, a singular point of $C_0$ either belongs to a multiple 
component, or is an ordinary double point, i.e. a node. 
%{\rm (}a Morse singularity{\rm )}.

Furthermore, if a node of $C_0$ is non-real, then it corresponds to 
a transverse intersection point of two different simple components of $C_0$.
If a node $p$ of $C_0$ is real, then $p$ is either a solitary node
{\rm (}given in local analytic coordinates by $x^2+y^2=0${\rm )}, 
or corresponds to a transverse intersection 
of two different real simple components of $C_0$. 
In the latter case, 
%for any $t \in (0, 1)$, 
%the two components of $C_0$ that intersect at $p$
%intersect only at $p$ and 
%\mnote{Why they intersect only at $p$?}
%$p$ is the unique intersection point of these two components of $C_0$;
$C_0$ is a union of two real curves intersecting only at $p$; %\mnote{It: a stronger statement}
the two real branches of 
%$C_0$ 
$\R C_0$ %\mnote{It: $\R C_0$ instead of $C_0$} 
at $p$ come from the same connected component of $\R C_t$, 
$t > 0$, under degeneration. 
\end{prop}
\begin{coro}\label{cK3}
If $C_0 \subset X$ is a reduced irreducible degeneration of simple Harnack curves, 
then all singular points of $C_0$ are solitary nodes. 
\end{coro}
%Recall that a solitary point $p\in S$ is a real singular point $p\in\R X$ such that every analytic
%branch of $S$ through $p$ is not real.
%In other words, the connected component of $\R S$ containing $p$ coincides with $\{p\}$.

\begin{proof}[Proof of Proposition \ref{hK3}]
%\mnote{Slava asks for references and explanations
%concerning membranes} 
%If $C \subset X$ is reduced, then
Away from multiple components,
each singular point $p \in C_0$ is isolated 
and (as a hypersurface singularity) can be described
through vanishing cycles on a curve $C_t$, $t>0$, which is a simple Harnack curve.
Each vanishing cycle $Z_t \subset C_t$ corresponds to a critical point of a morsification of $p$ 
({\em i.e.}, a holomorphic function with non-degenerate critical points 
%approximating 
which approximates the local 
equation of $C_0$ near $p$). Definition and properties of vanishing cycles can be found in \cite{Milnor}.
To find an appropriate collection of conjugation-invariant vanishing cycles we follow the procedure below.

Near a real singular point $p$ of $C_0$ the family of curves $C_t$ can be given as the zero set 
of a family of holomorphic functions 
$f_t:U\to\C$ on a small neighbourhood $p\in U\subset X$. 
Here $U$ can be chosen to be $\sigma$-invariant with the contractible real part $\R U=U\cap\R X$,
while 
$f_t$ can be chosen to be real ({\it i.e.}, $\conj \circ f_t = f_t \circ \sigma$, 
where $\conj: \C \to \C$
%commuting with 
is the complex conjugation).   
Multiplying by (non-vanishing on $U$) holomorphic functions if needed, we may assume that 
$f_t$, $t>0$, is a complex Morse function ({\it i.e.}, its critical points are isolated and have non-degenerate Hessians). 
Similarly, we may also assume for $f_t$ that the images of different critical points are
different and that the image of a non-real critical point is not real.

The multiplication trick also allows us to assume that the restriction $f_t|_{\R U}$ is a generic Morse function,
{\it i.e.}, that the stable and unstable manifolds for different critical points are transverse. 
Suppose that there exist two critical points with positive critical values, with indices different by 1, 
and such that 
%they 
these points are connected with a gradient trajectory. Then, such a gradient trajectory must be unique. 
Indeed, since the index of one of the critical point must be one, there could be not more than two such trajectories.
%But 
However, existence of two trajectories would imply a non-trivial mod 2 homology cycle in $\R U$ which is impossible. 
If two critical points with positive values are connected with a single trajectory, then these critical points are removable. 
Multiplying $f_t$
by an appropriate non-vanishing real function we can make such a pair of critical points into a complex conjugate pair.
Thus, inductively, we may assume that no pair of critical points with positive critical values
can be connected with a gradient trajectory. (Note that the points of indices 0 and 2 cannot be connected in this way,
since in the absence of trajectories to index 1 points it would imply an $S^2$-component for $\R U$.)

The critical points of $f_t$ for small $t>0$ can be thought of as the result of perturbation of the singular point $p$ for $f_0$.
% by means of $f_t$ of 
%the singular point $p$ of $f_0$.
In particular, the number of critical points of $f_t$ coincides 
with the Milnor number $\mu_p$ of the singularity $p$.
The set 
%critical values 
$\Pi_t\subset\C$ 
of critical values of $f_t$ 
%are 
is $\conj$-invariant and close to zero. 
Let us connect the points of $\Pi_t$ with $0$ by a $\conj$-invariant collection $\Gamma_p$
of $\mu_p$ smooth embedded paths (the paths connecting real points of $\Pi_t$ to $0$ may contain each other). 

Let $\gamma_t$ be one of these paths.
Its inverse image $f^{-1}_t(\gamma_t)\subset U$ is a hypersurface. A $\sigma$-anti-invariant K\"ahler symplectic form on $X$ 
has a 1-dimensional radical in the tangent space to $f^{-1}_t(\gamma_t)$.
With its help a tangent vector field 
%oriented  
to $\gamma_t$, oriented towards the critical value, canonically lifts to a vector field in $f^{-1}_t(\gamma_t)$ vanishing at a critical point of $f_t$. 
We define the membrane $M_t \subset X$ as the stable manifold of this point.
Since the critical points of $f_t$ are Morse, and no trajectories over our paths may connect critical points,
$M_t$ is an embedded disk.
The corresponding vanishing cycle $Z_t\subset C_t$ is the boundary of this disk $M_t$.
We have $M_t \cap C_t = Z_t$, while the membrane $M_t$ is never tangent to the curve $C_t$ along $Z_t$. 
It is an embedded disk in $X$ of self-intersection $-1$ 
(to define the self-intersection of a membrane we use a normal vector field to $Z_t$ in $C_t$
as the boundary framing).

For non-real singular points $p\in C_0$ the construction of the cycles $Z_t$ and the membranes $M_t$
is similar but locally we do not have to worry about the complex conjugation invariance.
Instead we use $\sigma(Z_t)$ and $\sigma(M_t)$ for the singular point $\sigma(p)\in C_0$.
%Instead we may ensure this
%property by 
%
%The cycle $Z_t \approx S^1$ bounds a membrane $M_t \subset X$,
%$M_t \cap C_t = Z_t$, $\dd M_t = Z_t$, which is an embedded disk in $X$ of self-intersection $-1$ 
%(to define the self-intersection of a membrane we use a normal vector field to $Z_t$ in $C_t$
%as the boundary framing).

%The membrane $M_t$ is never tangent to $\R C_t$ at the boundary.
%We choose a real morsification of $p$, so
For a given singular point $p \in C_0$, denote with $A_t \subset C_t$ the union of all vanishing cycles in $C_t$, 
and with $B_t \subset X$ 
the union of all their membranes. 
%
%are $\sigma$-invariant. 
%Each membrane is either invariant with respect to the complex conjugation on $X$
%or is exchanged with a membrane for another vanishing cycle.
%For a singular point $p\in S$ we denote with
%$K_p\subset S_t$ the union of the vanishing
%cycles at $p$
%and with $N_p\subset X$ the union of their membranes.
Both spaces $A_t$ and $B_t$ are connected.
Their union over all singular points of $C_0$ is  $\sigma$-invariant. 
% and with $N_p\subset X$ the union of their membranes.
%The Milnor fiber $M_p\subset S_t$ is a subsurface of $S_t$ that can be obtained
%as a regular neighborhood of $K_p$.
%The boundary components of $M_p$ correspond to analytic branches
%of $S$ at $p$.
%%The union $M_p\cup N_p$ must be contractible, so $K_p$ is connected while
%%two intersection
The vanishing cycles from
$A_t$ 
intersect transversely. 
Two cycles are either disjoint or intersect in a single point.
The dual graph of the vanishing cycles from $A_t$ cannot have cycles, see \cite{Milnor}. 
%(so that $C_t \cup B_t$ is contractible). 

Suppose that $p\notin\R C_0$. Then $A_t \subset C_t \setminus \R C_t$, 
but each component of $C_t \setminus \R C_t$ is of genus $0$ since $C_t$ is an M-curve. 
Thus, $A_t$ consists of a single vanishing cycle, and $p$ is a Morse point. 
Furthermore, $A_t \cup \sigma(A_t)$ 
separates $C_t$ into several connected components, %\mnote{It: why not three?}  
so $p$ must be a transverse intersection point of distinct components from $C_0$. %\mnote{It: $C_0$} 

Suppose that $p \in \R C_0$. Then, the tree of vanishing cycles of $A_t$ 
is $\sigma$-invariant, so it must have an invariant vertex or an invariant edge. 
However, an invariant edge would correspond to a transverse intersection of vanishing 
cycles. If these cycles are not real, then they intersect $\R C_t$ transversely 
in a single point which is impossible since $C_t \setminus\R C_t$ is disconnected. 
Thus, $A_t$ possesses at least one $\sigma$-invariant vanishing cycle $Z^r$
whose membrane $M^r$ is also $\sigma$-invariant. 

Suppose that $\sigma$ acts on $Z^r$ non-trivially. Then $Z^r\cap\R C_t$ consists
of two points, while $\gamma=M^r\cap\R X$ is a path connecting these points and transversal 
to $\R C_t$ at the endpoints. Let $M'\subset X$, $\dd M'\subset C_t$, 
be a small perturbation of the membrane $M^r$ such that $\dd M'$ and $\dd M^r$
are disjoint. Let $\gamma'=M'\cap\R X$. The parity of the self-intersection of $M^r$
coincides with the intersection number of $\gamma$ and $\gamma'$ since
all other points of  $M^r\cap M'$ come in pairs. This parity in its turn is determined
by the displacement of $\dd\gamma'\subset\R C_t$ with respect to $\dd\gamma\subset\R C_t$.
Let us enhance $\R C_t$ with the boundary orientation of one of the halves of
$C_t\setminus\R C_t$. 
Since $\dd M^r\cap\dd M'=\emptyset$, one of the points of $\dd\gamma$ must move in
the direction of this orientation, while the other one moves contrary to this direction.
Definition \ref{dH} implies that the intersection number of $\gamma$ and $\gamma'$ is even
whenever $\gamma$ connects two different components of $\R C_t$.
However, this is incompatible with the odd self-intersection of $M^r$. 

%Furthermore, 
If $\gamma$ connects
a component $L\subset\R C_t$ with itself, then
%$L$ must be non-contractible in
%a component $\Sigma\subset\R X$ such that $\Sigma\cap\R C_t=L$.
%In this case $Z^r$ 
$C_t\setminus Z^r$ is disconnected. 
Thus, %\mnote{It: sentence added} 
$C_0$ is a union of two real curves intersecting only at $p$.
%(since $Z^r$ is disjoint from $g$ cercles forming $\R C_t \setminus L$). 
%Therefore, 
Furthermore, $A_t=Z^r$, since otherwise there must be another cycle $Z'\subset A_t$
intersecting $Z^r$ transversally at a single point. In this case $p$ is an ordinary
double point with two real branches which is a transverse
intersection point of two distinct components of $C_0$. 

Any other real vanishing cycle $Z^r\subset C_t$ must be point-wise preserved by $\sigma$.
Suppose that $Z^r$ intersects another cycle $Z'$ in $A_t$.  
The cycle $Z'$ cannot be point-wise preserved since it intersects $Z^r$ at a single point.
Thus, $Z'$ is imaginary. 
Then, since
$Z'$ and $\sigma(Z')$ are transverse and $C_t \setminus \R C_t$ is disconnected,
$Z'\cap \sigma(Z')$ consists at least of two points which is impossible.
Therefore, any solitary real singular point $p \in C_0$ %\mnote{It: `solitary'?} 
has a unique vanishing cycle
corresponding to an oval of $\R C_t$, which implies that $p$ is 
%an ordinary solitary point. 
a solitary node. 
\end{proof}

\begin{rem}\label{rem-alt}
The proof of Proposition \ref{hK3} is based 
on concordance, ensured by a real vanishing cycle,
of complex orientations of a dividing 
real curve. 
% on components of real curves connected with
%ensured by a real vanishing cycle. 
This concordance is a well-known phenomenon in real algebraic geometry, 
responsible, in particular, for {\em Fieldler's orientation alternation}, see \cite{Viro-orFiedler}.
\end{rem}

\section{Deformations of 
%$M$-curves 
simple Harnack curves in K3-surfaces}
Assume that the surface $X$ is not hyperbolic. For a 
%smooth M-curve 
simple Harnack curve $C\subset X$ 
we choose an order on the components $L_0,\dots,L_g\subset\R C$, as well
as their orientations compatible with a half of $C\setminus\R C$.
Real curves in $X$ linearly equivalent to $C$ form the real part 
$\R |C|\approx \rp^g$ of the projective space $|C|\approx\pp^g$.
The homology classes $a_j=[L_j]\in H_1(C)$, $j = 1, \ldots, g$, form a 
%basis 
maximal collection of $a$-cycles 
making the map \eqref{mapIa} well-defined.
We denote with $\widetilde{\R\MM^a}$ the fixed locus of the involution induced by $\sigma$
on $\widetilde{\MM^a}$, the source of map \eqref{mapIa}.

%The space $\MM\cap\R |C|$ of smooth real curves in $\R |C|$ is usually disconnected.
%Consider an element $(C',\gamma)$ of the connected component of $\R\MM^a$ containing $C$. 
%The path $\gamma$ allows us to identify a component of $\R C'$ with one or 
Consider the subspace 
\[
\widetilde{\R\MM^i_C} \subset \widetilde{\R\MM^a} 
\]
consisting of pairs $(C',[\gamma])$, $\gamma:[0,1]\to \R |C|$, 
$\gamma(0)=C$, $\gamma(1)=C'$, %nodal real curves such that their nodes are disjoint .
where, for any $t \in [0, 1]$, 
the real curve $\gamma(t)$ is at worst nodal 
and 
%does not have real solitary nodes,
any non-singular real curve belonging to $\R |C|$ and sufficiently close to $\gamma(t)$
is simple Harnack. 
%We denote with $\R\MM^i_C\subset |C|$ the image of $\widetilde{\R\MM^i_C}$ under the forgetting map
%$(C',[\gamma])$
%\mnote{It is
%not a subspace of $\widetilde\R\MM^a$; if the definition is extended, some explanations
%are needed concerning vanishing cycles after going through singular curves;
%to move Lemma \ref{lem-appro} here?}  
%has $g$ non-singular components that are deformed by $\gamma$
%to the components $L_1,\dots,L_g\subset\R C$ so that they remain non-singular during
%the deformation.
%of smooth curves in $X$ linearly 
%equivalent to $C$ is usually disconnected.
%Denote with $\R\MM^i_C\subset\R\MM^i$ the component of $\R\MM^i$ containing
%a real curve $C\subset X$ and with $\widetilde{\R\MM^i_C}$ its universal covering.
%The space $\R\MM_C$ is called the {\em rigid isotopy class} of $C$.
%The space

\begin{rmk}\label{solitary_node}
If $(C',[\gamma]) \in \widetilde{\R\MM^i_C}$, then, for any $t \in [0, 1]$,
the curve $\gamma(t)$ does not have real solitary nodes.
\end{rmk}

\begin{prop}\label{non-m-smooth}  
Suppose that $L_j\subset\R C$ is not an $m$-component. 
Then, during the deformation $\gamma$, the component $L_j$ remains non-singular, 
i.e., we may consistently distinguish a smooth real component in $\gamma(t)$, $t\in[0,1]$,
coinciding with $L_j$ for $t=0$.
\end{prop}
\begin{proof}
%Let $t_0 \in (0, 1]$ be the smallest value corresponding to a singular curve in the deformation $\gamma$. 
Let us, first, show the statement assuming that $\gamma(t)$ is a non-singular curve
for any $t \in [0, 1]$. If $p \in \gamma(1)$ is a singular point, then
by Proposition \ref{hK3}, 
%the only possible real singularity of $\gamma(1)$ 
the point $p$ corresponds to 
a transversal intersection of two different real irreducible components of the reducible curve $\gamma(1)$,
and $p$ is a unique intersection point of these components.
Furthermore, the vanishing cycle of $p$ connects a real component $L(1 - \varepsilon)$ of $\gamma(1 - \varepsilon)$, 
$\varepsilon>0$, with itself.
Thus, $[L(1 - \varepsilon)] \ne 0 \in H_1(\R X; \Z_2)$, and 
Definition \ref{dH} implies that 
%$[L(1 - \varepsilon)] \ne 0 \in H_1(\R X; \Z_2)$
%and 
the non-spherical component $N \subset \R X$ 
containing $L(1 - \varepsilon)$ does not contain any other component of the real part of the curve $\gamma(1 - \varepsilon)$, 
that is, $L(1 - \varepsilon)$ is an $m$-component. 
This implies the statement of the proposition in the case considered. 
Moreover, 
% $p$ belongs to the non-spherical component $N$ of $\R X$
%and 
$g$ connected components of the real part of $\gamma(1 - \varepsilon)$
are contained in the union of spherical components of $\R X$. 
%$\bar K=K$ 
%for the component $K\subset \R X\setminus\R C$ adjacent to $L_j$.
%Thus, only an $m$-component may develop a singularity.

Consider now the general case. For any $t \in [0, 1]$, 
the curve $\gamma(t)$ is a degeneration of simple Harnack curves.
Thus, the particular case above implies that all singular points
of $\gamma(t)$ belong to $N$, 
%the non-spherical component $N$ of $\R X$
and $g$ connected components of the real part of $\gamma(t)$
are contained in the union of spherical components of $\R X$. 
This gives the statement required. 
\end{proof} 

%Assuming that $X$ is not hyperbolic we 
We may reorder the components of $\R C$
so that all components except possibly for $L_0$ are not $m$-components.
Thus, for any oriented component $L_j\subset\R C$, $j=1,\dots,g$, 
the map 
%\eqref{mapIz}
\eqref{mapIa} %\mnote{It: \eqref{mapIa} instead of \eqref{mapIz}} 
restricted to $\widetilde{\R\MM^i_C}$
induces the map 
\begin{equation}\label{rzmap}
\R I_{L_j}:\widetilde{\R\MM^i_C}\to\R.
\end{equation}
%since the vanishing cycles corresponding to imaginary singularities are
%disjoint from $\R C$. 

%For a smooth M-curve $C\subset X$
%we choose an order on its components $L_0,\dots,L_g$, as well
%as their orientations compatible with a half of $C\setminus\R C$,
We define 
\begin{equation}\label{rimap}
I_{\R C}=(\R I_{L_1},\dots,\R I_{L_g}):\widetilde{\R\MM^i_C}\to\R^{g}.
\end{equation}

%Let us choose one (of the two) complex orientation of $\R C$.
A component $L_j\subset\R C$, $j=0,\dots,g$, is either non-contractible
({\em i.e.},
%$[\R C]
$[L_j] \neq 0\in H_1(\R X;\Z_2)$) or such that 
$\Sigma_j\setminus L_j=\Sigma_j^+\cup\Sigma_j^-$ consists of two components.
Here, $\Sigma_j$ is the component of $\R X$ containing the component $L_j$. 
This component is oriented by the 2-form $\Omega$. We denote
with  $\Sigma_j^+$ the component of $\Sigma_j\setminus L_j$ whose boundary orientation
agrees with the chosen complex orientation of $\R C$ and with $\Sigma_j^-$ the other one.
Put 
\[
s_j^+=\int\limits_{\Sigma_j^+}\Omega=\ar(\Sigma_j^+),\ \ \ \ \ 
s_j^-=-\int\limits_{\Sigma_j^-}\Omega=-\ar(\Sigma_j^-).
\]
Clearly, $s_j^+-s_j^-=\ar(\Sigma_j)$.
If $L_j$ is non-contractible, we put $s_j^+=\infty$, $s_j^-=-\infty$.
Let
\[
\Delta=\{(x_1,\dots,x_g)\ |\ s_0^-<-\sum\limits_{j=1}^gx_j<s_0^+,\ s_j^-<x_j<s_j^+\}\subset\R^g.
\]

%The following proposition is an immediate corollary of Proposition \ref{intIa}.
\begin{prop}\label{inDelta}
%For an M-curve $C\subset X$ is a simple Harnack curve then 
%The image
%of $I_{\R C}$ is contained in the hyperplane
%\[
%H=\{\sum\limits_{j=0}^g x_g=0\}\subset \R^{g+1},
%\] 
%$\R\MM_S$ is contractible
%(in particular, $\widetilde{\R\MM_S}=\R\MM_S$) and
%and 
The map
%\begin{equation}\label{iRC}
$$I_{\R C}:\widetilde{\R\MM^i_C}\to \R^g$$
%\end{equation}
is a local diffeomorphism whose image is contained in $\Delta$.
%to the hyperplane
%$\{\sum\limits_{j=0}^g x_g=0\}\subset \R^g$.
%The image of $I_{\R S}$ is a diffeomorphism between $\R\MM_S$ and
%\[
%\Delta_S=\{(x_0,\dots,x_g)\ |\ \sum\limits_{j=0}^g x_g=0\}\cap
%\bigcu
%a diffeomorphism b 
\end{prop}
\begin{proof}
The map $I_{\R C}$ is a local diffeomorphism by Proposition \ref{intIa}.
Note that for a holomorphic curve $C'\subset X$
the area of a membrane whose boundary is contained in $C\cup C'$
depends only on the class of the membrane in $H_2(X,C\cup C')$.
In particular, to compute $I_{\R C}$ we may use the membranes contained in 
%$\R X\setminus N$. 
$\R X$. %\mnote{It: $\R X$} 
We have $s_j^-<\R I_{L_j}(C',\gamma)<s_j^+$ since the corresponding oval of $\R C'$
bounds two membranes of areas $s_j^+-\R I_{L_j}(C',\gamma)$
%\mnote{Slava asks for details
%concerning these membranes} 
and 
$\R I_{L_j}(C',\gamma)-s_j^-$, so these differences must be positive.
If $[L_0]=0\in H_1(\R X)$, then the corresponding ovals of the real curves from the deformation $\gamma$
cannot develop singularities by 
%Proposition \ref{hK3}.
Proposition \ref{non-m-smooth}. %\mnote{It: Proposition \ref{non-m-smooth} instead of Proposition \ref{hK3}} 
Since $[L_0]+\dots+[L_g]=0\in H_1(C)$, the corresponding oval of $C$ bounds
the membranes of area $s_0^++\sum\limits_{j=1}^g\R I_{L_j}(C',\gamma)$ and 
%$\sum\limits_{j=1}^g\R I_{L_j}(C',\gamma)-s_j^-$, 
$-\sum\limits_{j=1}^g\R I_{L_j}(C',\gamma)-s_0^-$, %\mnote{It: corrected} 
so these quantities are also positive.
\end{proof}

\section{Proof of Theorem \ref{main}} 
For any algebraic curve $D$ (not necessarily irreducible or reduced), the {\it multiplicity} of $D$
is the minimum among the multiplicities
of irreducible components of $D$. 

\begin{lem}\label{agenus}
Let $D \subset X$ be a real algebraic curve. 
Assume that $D$ is
either connected or consists of two connected non-real conjugated curves.
Put $d = [D] \in H_2(X; \Z)$, and denote by $k$ 
%the minimum among the multiplicities
%of irreducible components of 
the multiplicity of $D$. 
Then, the number of connected components of the real part $\R D$ of $D$
is at most $2+d^2/2k^2$.  
\end{lem} 

In particular, $2+d^2/2k^2\ge 0$ under the hypotheses of the lemma.

\begin{proof}
Assume, first, that $D$ is irreducible over $\R$ (but not necessarily reduced).
Let $D'$ be the reduced curve having the same set of points as $D$. 
The curve $D'$ is real and $[D] = k[D']$. The required inequalities are equivalent
for $D$ and $D'$. If $D'$ is irreducible over $\C$, the required inequality for $D'$
is a corollary of the Harnack inequality and the fact that the number of solitary real points of $D'$
is bounded from above by the difference between the arithmetic and geometric genera of $D'$. 
Suppose that $D'$ has two irreducible components over $\C$ (exchanged by the anti-holomorphic involution of $X$),
and denote these components by $D'_1$ and $D'_2$. Put $d'_1 = [D'_1] \in H_2(X; \Z)$
and $d'_2 = [D'_2] \in H_2(X; \Z)$.
We have
$$
\frac{[D']^2}{2} + 2 = \frac{(d'_1)^2}{2} + \frac{(d'_2)^2}{2} + 2 + d'_1d'_2 \geq d'_1d'_2, 
$$
since $(d'_i)^2 \geq -2$, $i = 1$, $2$. The number of real points of $D'$ is bounded from above by $d'_1d'_2$.
This implies the required inequality for $D'$, and thus, for $D$. 

Assume now that $D = D_1 \cup D_2$, where $D_1$ and $D_2$ are two real curves
without common components.
Assume, in addition, that each of these two curves is either connected or consists of two connected non-real conjugated curves, 
and that the required inequality is true for $D_1$ and $D_2$. 
Put $d_i = [D_i] \in H_2(X; \Z)$, $i = 1$, $2$,
and denote by $k_i$ 
%the minimum among the multiplicities of the irreducible components 
the multiplicity of $D_i$, $i = 1$, $2$. 
Denote by $n$ the number of intersection points of $D_1$ and $D_2$.
Suppose that $k_1 \leq k_2$. In this case, $k = k_1$. We have
\begin{equation}\label{decomposition}
\frac{d^2}{2k^2} + 2 \geq \frac{d_1^2}{2k_1^2} + \frac{d_2^2}{2k_2^2} + \frac{d_1\cdot d_2}{k_1^2} + 2
\geq \frac{d_1^2}{2k_1^2} + \frac{d_2^2}{2k_2^2} + \frac{nk_2}{k_1} + 2. 
\end{equation} 
If $n \geq 2$, then
$$
\frac{d_1^2}{2k_1^2} + \frac{d_2^2}{2k_2^2} + \frac{nk_2}{k_1} + 2 \geq 
(\frac{d_1^2}{2k_1^2} + 2) + (\frac{d_2^2}{2k_2^2} + 2),
$$ 
and the number of connected components of $\R D$ is at most $d^2/2k^2 + 2$. 
If $n = 1$, then 
$$
\frac{d_1^2}{2k_1^2} + \frac{d_2^2}{2k_2^2} + \frac{nk_2}{k_1} + 2 \geq 
(\frac{d_1^2}{2k_1^2} + 2) + (\frac{d_2^2}{2k_2^2} + 2) - 1. 
$$ 
In this case, the only intersection point of $D_1$ and $D_2$ is real,
and the number of connected components of $\R D$ is at most
$$
(\frac{d_1^2}{2k_1^2} + 2) + (\frac{d_2^2}{2k_2^2} + 2) - 1 \leq \frac{d^2}{2k^2} + 2.
$$ 
\end{proof} 

We say that a real algebraic curve $D \subset X$ is {\em $r$-maximal}
if $D$ is either connected or consists of two connected non-real conjugated curves,
and the number of connected components of the real part $\R D$ of $D$
is equal to $d^2/2k^2 + 2$, where $d = [D] \in H_2(X; \Z)$ and $k$ is the multiplicity of $D$. 

\begin{lem}\label{maximal_decomposition}
Let $D = D_1 \cup D_2 \subset X$ be an $r$-maximal real algebraic curve,
where $D_1$ and $D_2$ are real curves such that each of them
is either connected or consists of two connected non-real conjugated curves.
Then, the curves $D_1$ and $D_2$ are $r$-maximal and have the same multiplicity.
\end{lem}

\begin{proof}
Put $d_i = [D_i] \in H_2(X; \Z)$, $i = 1$, $2$,
and denote by $k_i$ 
the multiplicity of $D_i$, $i = 1$, $2$. 
Denote by $n$ the number of intersection points of $D_1$ and $D_2$.
The $r$-maximality of $D$ and the inequalities (\ref{decomposition}) 
imply that $k_1 = k_2$ and the curves $D_1$ and $D_2$ are $r$-maximal.
\end{proof}

Recall that in the linear system of genus $g$ polarizing the K3-surface $X$ 
we may choose a curve passing through arbitrary $g$ points. 
Let $C$ be the real curve passing through $g$ points on $g$ distinct spherical 
components $\Sigma_1$, $\ldots$, $\Sigma_g$ of $\R X$.
%(including the one with the maximal area). 
Thus, $\R C$ contains at least $g$ components $L_1$, $\ldots$, $L_g$ 
at these spherical components. 
Slightly perturbing the curve $C$ if needed we may assume that $C$ is smooth.
Since the polarization is non-contractible, the real locus $\R C$ must also 
contain a non-contractible component $L_0\subset\R C$ at the non-contractible 
component
$N\subset\R X$.
Thus, $C$ is a simple Harnack curve. 

\begin{lem}\label{reduced}
Let $C' \in \R|C|$ be a connected curve 
%having a real component on 
intersecting each connected component
$\Sigma_1$, $\ldots$, $\Sigma_g$. If $g \geq 2$, the curve $C'$ is reduced.
%Furthermore, any singular point $x \in C'$ belonging to a spherical component
%of $\R X$ is solitary double. 
\end{lem} 

\begin{proof}
Note that 
%$C'$ 
%has 
%intersects at least $g + 1$ 
%real 
%connected components of $\R X$, 
%since 
%it 
$C'$ necessarily %has a real component on 
intersects $N$. In addition, $[C']^2 = [C]^2 = 2g - 2 > 0$ (since $g \geq 2$). 
Thus, 
%Lemma \ref{agenus} implies that the real part $\R C'$ of $C'$ 
%has exactly $g + 1$ connected components. 
$C'$ is $r$-maximal and of multiplicity $1$. 
Lemma \ref{maximal_decomposition} implies that all irreducible components of $C'$ 
are of multiplicity $1$, that is, $C'$ is reduced.  
\end{proof} 

Assume that $g \geq 2$. Choose a complex orientation of $\R C$.
For every $j = 1$, $\ldots$, $g$, 
the connected component $\Sigma_j$ is oriented by the 2-form $\Omega$
and is divided by the oval $L_j$ in two disks 
$\Sigma^+_j$ and $\Sigma^-_j$, where $\Sigma^+_j$ is the disk whose boundary orientation
agrees with the chosen complex orientation of $\R C$.
Denote by $s^+_j$ and $-s^-_j$ the areas of the disks $\Sigma^+_j$ and $\Sigma^-_j$.

Lemma \ref{reduced} and Propositions \ref{hK3}, \ref{inDelta} 
imply that the inverse image of the line 
\[
\{ (\frac{s_1^++s_1^-}2,\dots,\frac{s_{g-1}^++s_{g-1}^-}2,u)\ |\ u\in\R\}\subset\R^g
\]
under the 
%diffeomorphism \eqref{toDelta}.
map 
$$I_{\R C}:\widetilde{\R\MM^i_C}\to \R^g$$  
is a segment $\SSS \subset \widetilde{\R\MM^i_C}$ 
whose closure in $\widetilde{\R\MM^a}$ %\mnote{It: `closure in $\widetilde{\R\MM^a}$'} 
has two extremal points 
%two ends 
corresponding to 
%two 
nodal 
%degenerations 
curves $C_+$ and $C_-$; 
%of simple Harnack curves;
each of the curves $C_+$ and $C_-$ has a solitary double point in $\Sigma_g$. 
%whose only singularity is an isolated node at $S$ by Corollary \ref{cK3}.
Furthermore, according to Proposition \ref{non-m-smooth},
the curves $C_+$ and $C_-$ do not have other singular points 
in $\R X \setminus N$, and any curve corresponding to a 
%non-extremal 
point of $\SSS$ %\mnote{It: `non-extremal' removed} 
does not have singular points in $\R X \setminus N$. 
Let $D'$ and $D''$ be any two curves corresponding to distinct points of $\SSS$. 
The ovals of $D'$ and $D''$ at the spherical components $\Sigma_1$, $\ldots$, $\Sigma_{g - 1}$
divide
the corresponding spheres into disks of equal areas,
so they intersect at least at $2(g-1)$ points at these components. 
By the B\'ezout theorem, $D' \cap D'' \cap (N \cup \Sigma_g)=\emptyset$.
Thus, $(C_+\cap N)\cup(C_-\cap N)$ bounds a proper compact subsurface of $N$ of area
$s_g^+-s_g^-=\ar(\Sigma_g)$. 
%which implies $\ar(N)>\ar(S)$. 

%Lemma \ref{non-intersect} implies that $N$ contains a proper compact region whose area
%is equal to the area of $\Sigma_g$. For a sufficiently small real number $\epsilon > 0$,
%corresponding curve $C_\varepsilon$ and paths $\gamma^\pm_\varepsilon$, 
%%and $\gamma^-$,
%such a region $R_g$ is the union of the real components in $N$ of the curves $\gamma^+_\varepsilon(t)$, $t \in [0, \tau^+]$, 
%and $\gamma^-_\varepsilon(t)$, $t \in (0, \tau^+]$. 

Similarly, for every connected component $\Sigma \subset \R X$ 
different from $\Sigma_1$, $\ldots$, $\Sigma_{g - 1}$, $N$, there is
a proper compact subsurface $R \subset N$ whose area
is equal to the area of $\Sigma$. 
The B\'ezout theorem implies that all these subsurfaces $R$ are pairwise disjoint.
This proves the statement of the theorem in the case $g \geq 2$. 

Assume now that $g = 1$. In this case, the linear system $|C|$ is $1$-dimensional, 
$\R |C| \simeq \rp^1$, and the previous arguments can be easily adapted.
%The curves of $|C|$ are pairwise disjoint.
Through any point of $X$ one can trace a unique curve belonging to $|C|$. 
This defines a projection $\pi_\R: \R X \to \rp^1$. 
Note that $\pi_\R(N)$ coincides with $\rp^1$. 
The image under $\pi_\R$ of any spherical component $\Sigma_j$ of $\R X$ is a closed segment,
and all such segments are pairwise disjoint. 
Each segment $\pi_\R(\Sigma_j)$ gives rise to a proper compact subsurface in $N$ 
(the intersection of $N$ with the inverse image under $\pi_\R$ of the segment)
whose area is equal to the area of $\Sigma_j$, 
and all these subsurfaces are pairwise disjoint.
\qed 

%\begin{rmk}
%The proof of Theorem \ref{main} uses minimal Harnack curves with an $m$-component
%which always exist under the hypotheses of the theorem.
%Proposition \ref{single} implies that the topology of such curves is very restricted: 
%no two ovals of the same curve may belong to the same component of $\R X$.
%In particular, the number of components of such curves is not greater than 10,
%and thus its genus is not greater than 9.
%%%It might be interesting to study other Harnack curves on K3-surfaces.
%%Below we formulate a few open (to the best of our knowledge) questions 
%%on existence of other simple Harnack curves.
%%% that are either non-minimal, or
%%%without am $m$-component.
%\end{rmk} 

%\vskip10pt 

\section{Simple Harnack curves in K3 surfaces: further directions and questions}\label{simple-Harnack}
For the proof of Theorem \ref{main} we have used simple Harnack curves $C\subset X$
of rather special type: each component of $\R X$ contained not more than one
component of the curve $\R C$. Under this assumption a real curve is a simple Harnack
whenever it is an M-curve.

We finish the paper by taking a look at more general simple Harnack curves. 
Namely, we assume that $C\subset X$ is a simple Harnack curve,
and $X$ is a real K3-surface which is not hyperbolic.
Then the locus of the M-curve
$\R C$ has not more than one $m$-component.
We order the components $L_j$, $j=0,\dots,g$ of $\R C$ so that
all of them, except possibly $L_0$, are not $m$-components.
%All components of $\R C$ except possibly for $L_0$ are 
%not $m$-components, and 
Thus the map $I_{\R C}$ from Proposition \ref{inDelta} is well-defined.
%For the rest of the paper we consider the situation when
%a simple Harnack curve cannot degenerate to a reducible curve.

Note that for a degeneration $C_t$, $t\in [0,1]$, of simple Harnack curves such that $C_1=C$,
any curve $C_t$, $t>0$, is naturally identified with a point in $\widetilde{\R\MM^i_C}$.

\begin{defn}\label{def-min}
A simple Harnack curve $C\subset X$ is called {\em minimal}
if for any degeneration  $C_t$, $t\in [0,1]$, of simple Harnack curves  such that $C_1=C$
%continuous family $C_t\subset X$, $t\in [0,1]$, of curves
%from $|C|$ such that $C_1=C$, $C_t$, $t>0$, is a simple Harnack curve and
and $\lim\limits_{t\to 0} I_{\R C}(C_t)\in \Delta$, %\mnote{It: what is $I_{\R C}(C_t)$?} 
the curve $C_0$ is reduced and irreducible.
\end{defn}

\begin{exa}
%If $C \subset X$ is an M-curve of genus $g \geq 2$ such that
%every connected component of $\R X$ contains at most one connected
%component of the real part $\R C$ of $C$, Lemma \ref{reduced} implies
%that $C$ is a minimal simple Harnack curve.
%
If $X$ does not contain embedded curves of genus less than 
%the genus $g$
$g \geq 2$  
(in particular, it does not contain $(-2)$-curves), then any simple Harnack curve
of genus $g$ is a minimal simple Harnack curve.
%\mnote{Slava asks for a proof} 
\end{exa}

\begin{prop}\label{prop-diffeo}
If $C$ is a minimal simple Harnack curve, then the map
\begin{equation}\label{toDelta}
I_{\R C}:\widetilde{\R\MM^i_C}\to\Delta 
\end{equation}
from Proposition \ref{inDelta} is a diffeomorphism.
\end{prop}
\begin{proof}
Since $\Delta$ is simply connected and $I_{\R C}$ is a local diffeomorphism by Proposition \ref{inDelta},
it suffices to prove that \eqref{toDelta} is proper.
The space $\widetilde{\R\MM^i_C}$ is an open manifold covering a subset of $|C|$. 
By Proposition \ref{hK3},  for any degeneration $C_t$, $t\in [0,1]$,
of simple Harnack curves such that $C_1=C$ and $\lim\limits_{t\to 0} I_{\R C}(C_t)\in\Delta$,
%from
%$\widetilde{\R\MM^i_C}$,
the limit curve $C_0$ is smooth.
%if $I_{\R C}(C_0)\in\Delta$.\mnote{It: we have a problem with terminology here, because simple Harnack curves
%are not elements of $\widetilde{\R\MM^i_C}$} 
\ignore{
%By Proposition \ref{hK3}, 
If \eqref{toDelta} is not proper, then there exists a degeneration $C_t$, $t\in [0,1)$,
of simple Harnack curves 
from
$\widetilde{\R\MM^i_C}$
such that $\lim\limits_{t\to 0} I_{\R C}(C_t)\in \Delta$ but with $C_0\notin\widetilde{\R\MM^i_C}$.
%Note that all curves from $\widetilde{\R\MM^i_C}$ can be approximated by simple Harnack curves,
%as smoothing the 
Definition \ref{def-min} implies that $C_0$ is reduced.
Then
Proposition \ref{hK3} implies that $C_0\in \widetilde{\R\MM^i_C}$.
}
\end{proof}

%In this subsection we assume that $C\subset X$ (with $\R C=L_0\cup L_1\cup\dots\cup L_g$)
%is a minimal simple Harnack curve.

%If $[L_j]=0\in H_1(\R X;\Z_2)$ then it divides the component $\Sigma_j\subset\R X$
%into 
We say that the components $L_{j_-}, L_{j_0}, L_{j_+}$ {\em nest} if
they are contractible ({\em i.e.}, $[L_{j_-}]=[L_{j_0}]=[L_{j_+}]=0\in H_1(\R X;\Z_2)$),
belong to the same component  $\Sigma_j\subset\R X$, and
one component of $\Sigma_j\setminus L_{j_0}$ contains $L_{j_-}$ while the
other one contains $L_{j_+}$, see Figure \ref{fig-nest}.
\begin{figure}[h]
\includegraphics[height=45mm]{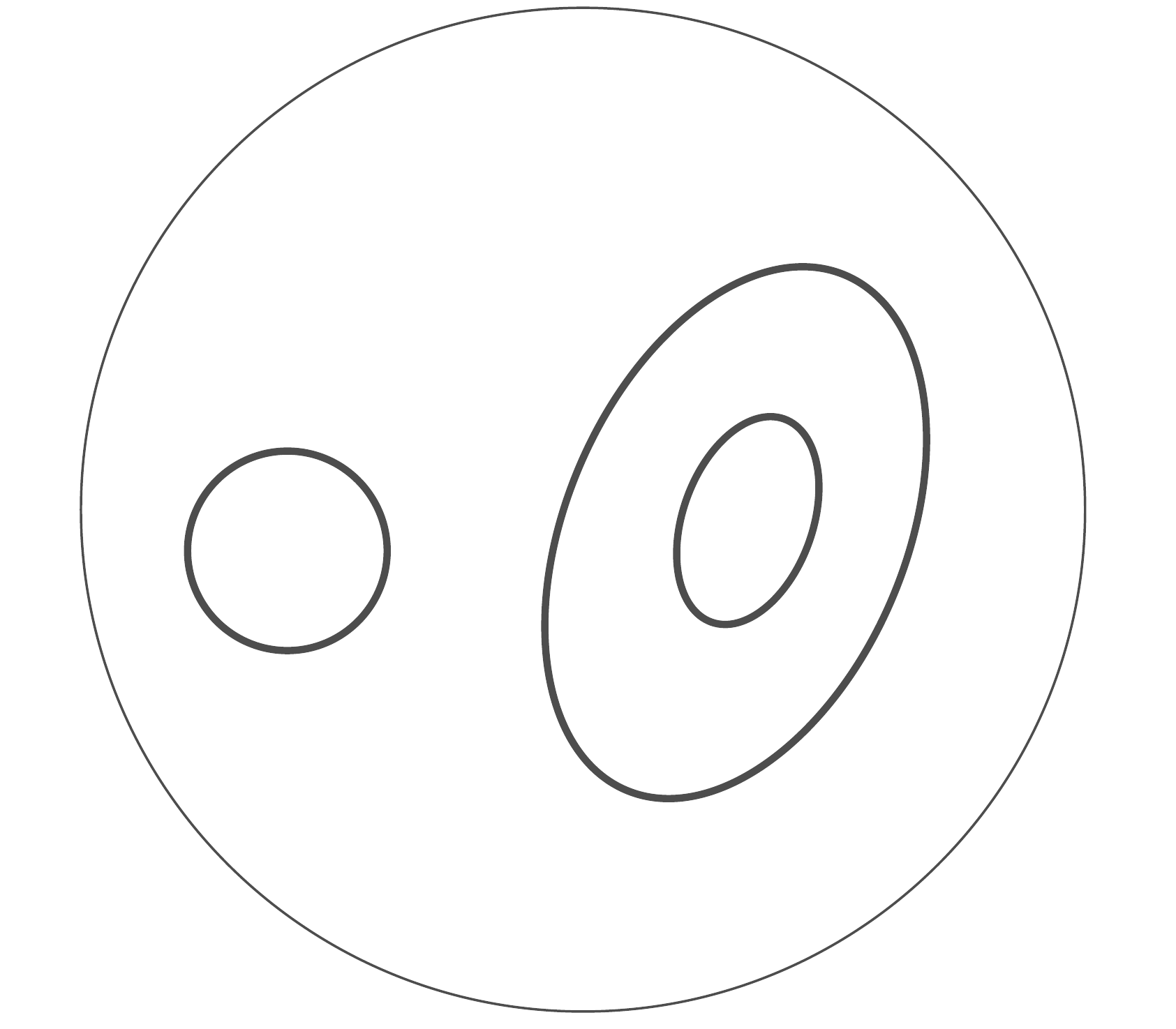}%\vspace{-10pt}
\caption{Three nesting ovals in a spherical component of $\R X$.
\label{fig-nest}}
\end{figure}

\begin{prop}\label{nest}
No three components of a minimal simple Harnack curve $C$ can {nest}.
\end{prop}
\begin{proof}
Passing to different nesting components if needed we may assume that
$L_{j_+}$ and $L_{j_0}$ (resp. $L_{j_+}$ and $L_{j_0}$) are
adjacent to the same component $K_{j_+}\subset\R X\setminus \R C$
(resp.  $K_{j_-}\subset\R X\setminus \R C$).
Then Definition \ref{dH} implies that the complex orientations of $L_{j_-}, L_{j_0}, L_{j_+}$ alternate.
Renumbering the components of $\R C$ if needed we may assume that $L_j^+=L_{g}$
and $L_j^-=L_{g-1}$. Also, we may assume that the boundary orientation of $\dd K_{j_+}$
induced by $\Omega$ agree with the complex orientation of $\R C$.

Consider the inverse image of the interval
\[
\{ (0,\dots,0,-u,u)\in \Delta\ |\ 0\le u< s^+_g\}\subset\Delta
\]
under the diffeomorphism \eqref{toDelta}.
It corresponds to the 
%curves from 
the elements of $\widetilde{\R\MM^i_C}$ such that the area of $K_{j_+}$
becomes smaller by $u$. For $$u>\ar(K_{j_+})+\ar(K_{j_-})<s^+_j$$ we get self-contradicting
conditions for the resulting smooth curve in $\Sigma_j$.
\end{proof}

\begin{prop}\label{single}
If a minimal simple Harnack curve $C$ has an $m$-component,
%If $L_0\subset\R C$ is an $m$-component of a minimal simple Harnack curve
%$\R C$,
%\mnote{to write `If a minimal simple Harnack curve $C$ has an $m$-component'?} 
then each component $\Sigma\subset\R X$ contains not more than one
component of $\R C$.
\end{prop}
\begin{proof}
Let $L_0$ be the $m$-component.
%By the definition of the $m$-component the 
%If $\Sigma\supset L_0$, 
Then, the component $N\subset \R X$ containing $L_0$ is disjoint from the other components
of $\R C$ by the definition of the $m$-component.
Since $L_0$ is not contractible, we have $s_0^-=-\infty$, $s_0^+=+\infty$,
thus $\Delta$ is a cube.
Suppose that %$\Sigma$ contains two components, say $L_1$ and $L_2$.
a component $K\subset \R X\setminus \R C$ is adjacent to $L_1$ and $L_2$.
Considering the inverse image of the line $\{(0,u,0,\dots,0)\ |\ u\in\R\}\subset\R^g$
under \eqref{toDelta} we get a contradiction at the value $u=\pm \ar(K)$ as in 
the proof of Proposition \ref{nest}.
\end{proof}

\begin{prop}
If a minimal simple Harnack curve $C$ does not have an $m$-component,
but has a non-contractible component $L$ contained in the component 
$N\subset\R X$, then $N$ contain another non-contractible component $L'\subset\R C$
homologous to $L$.
Furthermore, in this case we have $\R C\cap N=L\cup L'$.
\end{prop}
\begin{proof}
If $(N\cap\R C)\setminus L$ consists of contractible components, then we have a
contradiction with  Definition \ref{dH}. Thus, there must be another non-contractible component 
$L'\subset \R C$ on $N$.
%\mnote{Why not several components
%whose union is homologous to $L$?} 
Without loss of generality, we may assume that $L=L_0$,
$L'=L_1$.
If there exists yet another component $L''\subset\R C$
on $N$ (contractible or not), then considering the inverse image of the
ray $\{(u,0,\dots,0)\ |\ u\ge 0\}\subset\R^g$ under \eqref{toDelta}
we get a contradiction as in the proofs of Propositions \ref{nest} and \ref{single}.
If $L$ and $L'$ are the only components on $N$, then they must be homologous
by Definition \ref{dH}.
\end{proof}

The following example shows that a simple Harnack curve can have two homologous non-contractible components. 

\begin{exa}\label{exa2hom}
Recall that a real K3-surface $X$ polarized by genus 2 admits
a double covering $\pi:X\to\pp^2$ branching along a real curve
$B\subset\pp^2$ of degree 6. Denote with $\rho:X\to X$ the involution of
deck transformation of $\pi$.
Since $\pi$ is a real map, 
the holomorphic involutions $\rho$ and the anti-holomorphic involution $\sigma$ 
commute. Therefore, the composition $\rho\circ\sigma$ is an antiholomorphic
involution on $X$. Denote with $\R X'$ its fixed locus. 

\begin{figure}[h]
\vspace{-10pt}
\includegraphics[height=75mm]{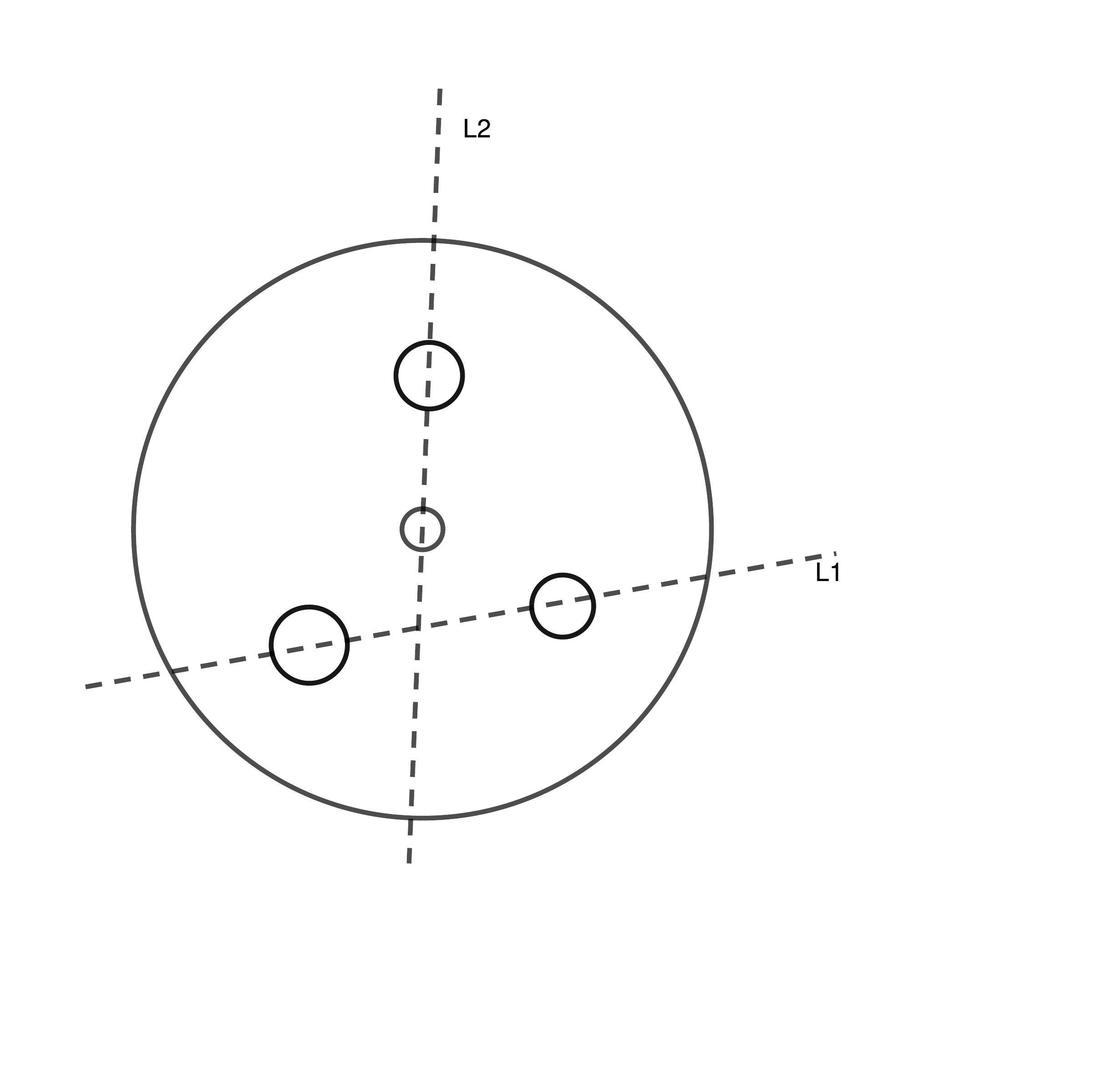}\vspace{-50pt}
\caption{Two hyperplane sections in a genus 2 polarized K3-surface. 
\label{C6}}
\end{figure}
Let $X$ be the K3-surface obtained as the double covering $\pi:X\to\pp^2$
branched along a real sextic $B \subset\pp^2$ whose real locus
is depicted at Figure \ref{C6}
(more precisely, the figure shows the isotopy type of the real locus and the position of ovals of the curve
with respect to two auxiliary straight lines). 
The involution of complex conjugation on 
%$\cp^2$
$\pp^2$ 
can be lifted to $X$ in two ways differing by the deck transformation $\rho$.
We may assume that $\R X$ covers the non-orientable half of $\rp^2\setminus\R B$.
Then, $\R X$ is homeomorphic to the disjoint union of a torus and four spheres, 
while $\R X'$ is homeomorphic to a surface of genus 4. 

The equation $z^2=f(x_0,x_1,x_2)$, where $f$ is a homogeneous
polynomial defining the curve $B$, gives an embedding of 
$X$ into the weighted projective space $\pp(1,1,1,3)$.
Accordingly, the inverse image $A_L=\pi^{-1}(L)$ of a real line $L\subset\pp^2$
sits in the weighted projective plane $\pp(1,1,3)$ embedded in $\pp(1,1,1,3)$.
Suppose that $L$ intersects $B$ in 6 distinct real points. 
Then, 
$\R A_L=A_L\cap\R X$ is an M-curve.
Its three ovals in the real part of $\pp(1,1,3)$ have alternate complex orientations
(with respect to the pencil of lines passing through the singular point of $\pp(1,1,3)$) 
by Fiedler's orientation alternation, cf. Remark \ref{rem-alt}. 
Furthermore, $\R A_L'=A_L\cap\R X'$ is also an M-curve,
while complex orientations of $\R A_L$ and $\R A_L'$ agree over
any point of $B\cap A_L$ after multiplication by $i$.

Consider the lines $L_1$ and $L_2$ from Figure \ref{C6}. %Its real part 
%The inverse image $A_L=\pi^{-1}(L)$ of a real line $L\subset\pp^2$ transverse
%to $B$ is a smooth curve of genus 2 (from the polarization of $X$).
%If $\R L\cap\R B$ consists of 6 points then $\R A_L=\R X\cap A_L$ 
%consists of 3 ovals, so that $A_L$ is an M-curve.
%The curve $A_L$ ...
%Similarly,
%$\R A_L'=\R X'\cap A_L$ is an M-curve. 
The intersection $A_{L_1}\cap A_{L_2}$ consists of two points from $\R X'$. 
The complement $(A_{L_1} \cup A_{L_2})\setminus (\R A_{L_1} \cup \R A_{L_2})$ 
consists of two connected components each obtained as a bouquet of
two planar domains.
Accordingly, the real curve $C$ obtained by a small perturbation of
$A_{L_1} \cup A_{L_2}$ is an M-curve. Its real part $\R C=C\cap\R X$ 
has two non-contractible ovals at the torus component of $\R X$
and an oval at each spherical component of $\R X$.
The complex orientation of $\R C$ is determined by the complex
orientations of $\R A_{L_j}$, $j=1,2$, chosen so that the corresponding
orientations at $\R A'_{L_j}$ are associated to the intersecting halves 
of $A_{L_j} \setminus \R A_{L_j}$. Thus, the complex orientations of the non-contractible
ovals of $\R C$ are opposite, and $C$ is a simple Harnack curve. 

For other polarizations of $X$, existence of 
simple Harnack curves with two homologous non-contractible components 
remains unknown to us. 
\end{exa} 

%Below we formulate 
%a few 
In the end, we formulate some open (to the best of our knowledge) questions 
%on existence of other 
concerning simple Harnack curves. 

\begin{que}
Does there exist a non-minimal simple Harnack curve? %\mnote{It: with the new definition,
%we can, probably, construct an example?}  
%In other words, can a simple Harnack curve degenerate to a non-reduced curve
%without contraction of one of its ovals? 
\end{que} 

\begin{que}
Does there exist a K3-surface containing
%a simple Harnack curve of arbitrary large genus on a (closed) K3-surface? 
simple Harnack curves of arbitrary large genus? 
\end{que} 

\begin{que}
%Does any non-empty real projective 
Which K3-surfaces contain simple Harnack curves? 
\end{que}

%\begin{que}\label{2hom}
%Does there exist a simple Harnack curve with two homologous non-contractible components?
%\end{que}
%
%\begin{rmk}
%After finishing the first version of the paper we have found that the answer to 
%question \ref{2hom} is positive, see Example \ref{exa2hom},
%at least in the case when $X$ is genus 2 polarized.
%Existence of such curves for other polarization of $X$ remains unknown to us.
%\end{rmk}
% 
 
%Since $\lim\limits_{t\to 0} I_{\R C}(C_t)\in\Delta$ no oval of 
%$\R C$ could shrink, in particular,
%$D$ cannot have solitary nodes. Also $D$ cannot
%have imaginary singularities since $C_t$ is an M-curve.
%As in the proof of Proposition \ref{hK3} any real singularity of 

\bibliography{b}
\bibliographystyle{plain}

\end{document}